\documentclass[11pt,leqno]{article}
\usepackage{amsthm,amsfonts,amssymb,amsmath,oldgerm}
\usepackage{epsfig}
\numberwithin{equation}{section}
\usepackage[thinlines]{easybmat}

\def\eps{\varepsilon }

\def\eps{\varepsilon}

\newcommand\br{\begin{remark}}
\newcommand\er{\end{remark}}
\newcommand\bp{\begin{pmatrix}}
\newcommand\ep{\end{pmatrix}}
\newcommand\be{\begin{equation}}
\newcommand\ee{\end{equation}}
\newcommand\ba{\begin{equation}\begin{aligned}}
\newcommand\ea{\end{aligned}\end{equation}}

\newcommand{\bap}{\begin{app}}
\newcommand{\eap}{\end{app}}
\newcommand{\begs}{\begin{exams}}
\newcommand{\eegs}{\end{exams}}
\newcommand{\beg}{\begin{example}}
\newcommand{\eeg}{\end{exaplem}}
\newcommand{\bpr}{\begin{proposition}}
\newcommand{\epr}{\end{proposition}}
\newcommand{\bt}{\begin{theorem}}
\newcommand{\et}{\end{theorem}}
\newcommand{\bc}{\begin{corollary}}
\newcommand{\ec}{\end{corollary}}
\newcommand{\bl}{\begin{lemma}}
\newcommand{\el}{\end{lemma}}
\newcommand{\bd}{\begin{definition}}
\newcommand{\ed}{\end{definition}}
\newcommand{\brs}{\begin{remarks}}
\newcommand{\ers}{\end{remarks}}

\newcommand{\sign}{{\text{\rm sgn }}}

\newtheorem{theo}{Theorem}[section]

\newtheorem{exams}[theo]{Examples}

\numberwithin{equation}{section}
\newcommand{\MM}{{\mathbb M}}

\newtheorem{theorem}{Theorem}[section]
\newtheorem{proposition}[theorem]{Proposition}
\newtheorem{corollary}[theorem]{Corollary}
\newtheorem{lemma}[theorem]{Lemma}
\newtheorem{definition}[theorem]{Definition}

\newtheorem{example}[theorem]{Example}
\newtheorem{remark}[theorem]{Remark}

\newtheorem{thm}{Theorem}
\newtheorem{corr}{Corollary}

\newcommand{\RM}{\mathbb{R}}
\newcommand{\ZM}{\mathbb{Z}}

\newcommand{\CM}{\mathbb{C}}

\newcommand{\HM}{\,\mbox{\bf H}}

\newcommand{\cn}{\operatorname{cn}}

\newcommand{\spec}{\operatorname{spec}}

\newcommand{\disc}{\operatorname{disc}}

\title{%Low-Frequency  Stability Analysis of Periodic Traveling Wave Solutions of the Generalized Korteweg-de Vries Equation%: Rigorous Justification of the Whitham Modulation Equations
Rigorous Justification of the Whitham Modulation Equations for the Generalized Korteweg-de Vries Equation
}

\author{\sc \small
Mathew A. Johnson\thanks{Indiana University, Bloomington, IN 47405;
matjohn@indiana.edu: Research of M.J. was partially supported by an NSF Postdoctoral Fellowship under NSF grant DMS-0902192.}
~~~~~
Kevin Zumbrun\thanks{Indiana University, Bloomington, IN 47405;
kzumbrun@indiana.edu:
Research of K.Z. was partially supported
under NSF grants no. DMS-0300487 and DMS-0801745.
 }}
\begin{document}

\maketitle

%%%%%%%%%%%%%%%%%%%%%%%%%%%%%%%%%%%%%%%%%%%%%%%%%%%%%%%%%%%%%%%%%%%%%%%%%%%%%%%%%%%%%%%%%%%%%

\begin{center}
\end{center}

%%%%%%%%%%%%%%%%%%%%%%%%%%%%%%%%%%%%%%%%%%%%%%%%%%%%%%%%%%%%%%%%%%%%%%%%%%%%%%%%%%%%%%%%%%%%%

\begin{abstract}
In this paper, we consider the spectral stability of spatially periodic traveling wave solutions of the generalized Korteweg-de Vries
equation to long-wavelength perturbations.  Specifically, we extend the work of Bronski and Johnson by demonstrating
that the homogenized  system describing the mean behavior of a slow modulation (WKB) approximation of the solution correctly
describes the linearized dispersion relation near zero frequency of the linearized equations about the
background periodic wave.  The latter has been shown by rigorous Evans function techniques %, i.e. Floquet theory,
to control the spectral stability near the origin, i.e. stability to slow modulations of the underlying solution.
In particular, through our derivation of the WKB approximation we generalize the modulation expansion of Whitham for the KdV to a more general class of
equations which admit periodic waves with nonzero mean.  As a consequence, we will show
that, %we provide a simple derivation of the full Whitham system for the generalized Korteweg-de Vries
%equation analogous to that provided by Whitham for the KdV case and show that,
assuming a particular non-degeneracy condition, % the Whitham system is of evolutionary type,
spectral stability near the origin is equivalent with the local well-posedness of the Whitham system.
%
%
%To this end, we use a slow modulation (WKB approximationwe provide a simple derivation of the Whitham system
%for the generalized Korteweg-de Vries equation analogous to that derived by Whitham for the KdV
%
%
%To this end, we give a simple derivation of the
%
%
%we carry out a formal Whtiham modulation theory calculation
%and find that a necessary and sufficient condition for such
%a modulational stability is the hyperbolicity of the homogenized
%equations.  This equivalence is determined by comparing the formal calculation to the rigorous modulation theory calculation
%via the Evans function, i.e. Floquet theory, recently carried out by Bronski and Johnson.
\end{abstract}

%\pagestyle{myheadings}
%\thispagestyle{plain}
%\markboth{Mathew A. Johnson Kevin Zumbrun}{Transverse Instability of gKdV}

%\clearpage
%\tableofcontents
%\clearpage
%%%%%%%%%%%%%%%%%

\bigbreak
\section{Introduction}

In the area of nonlinear dispersive waves the question of stability is of fundamental importance as it determines those
solutions which are most likely to be observed in physical applications.  In this paper, we consider the
stability of the spatially periodic traveling wave solutions of the generalized Korteweg-de Vries (gKdV) equation
\begin{equation}
u_t=u_{xxx}+f(u)_x\label{gkdv}
\end{equation}
where $u$ is a scalar, $x,t\in\RM$ and $f$ is a suitable nonlinearity.  Such equations arise in a variety of applications.
For example, the case $f(u)=u^2$ corresponds to the well known
Korteweg-de Vries (KdV) equation which arises as
a canonical model of weakly dispersive nonlinear wave propagation \cite{KdV} \cite{SH}.  Moreover, the cases $f(u)=\beta u^3$, $\beta=\pm 1$, corresponds
to the modified KdV equation which arises as a model for large amplitude internal waves in a density stratified medium, as well
as a model for Fermi--Pasta-Ulam lattices with bistable nonlinearity \cite{BM} \cite{BO}.  In each of these
two cases, the corresponding PDE can be realized as a compatibility condition for a particular Lax pair and hence
the corresponding Cauchy problem can (in principle) be completely solved via the famous inverse scattering
transform\footnote{More precisely, the Cauchy problem for equation \eqref{gkdv} can be solved via the inverse
scattering transform if and only if $f$ is a cubic polynomial.}.  However, there are a variety of applications
in which equations of form \eqref{gkdv} arise which are not completely integrable and hence the inverse scattering
transform can not be applied.  For example, in plasma physics equations of the form \eqref{gkdv}
arise with a wide variety of power-law nonlinearities depending on particular physical considerations \cite{KD} \cite{M} \cite{MS}.
Thus, in order to accommodate as many applications as possible the methods employed in this paper will not rely
on complete integrability of the PDE \eqref{gkdv}.  Instead, we will make use of the integrability of the ordinary
differential equation governing the traveling wave profiles: this ODE is always Hamiltonian, regardless
of the integrability of the corresponding PDE.

The stability of such solutions has received much attention as of recently:
for example, see
\cite{BD}, \cite{BrJ}, \cite{BrJK}, \cite{DK}, \cite{G1}, \cite{G2}, \cite{HK}, \cite{J1}.  Here, we are interested
in the spectral stability to long wavelength perturbations, i.e. to slow modulations of the underlying wave.  % background periodic solution.
There is a well developed (formal) physical theory for dealing with such problems which is known as Whitham modulation
theory \cite{W1} \cite{W2}.  This formal theory proceeds by rescaling the governing PDE via the change of variables $(x,t)\mapsto(\eps x,\eps t)$
%to yield the equation
%\[
%u_t=f(u)_x+\eps^2u_{xxx}.
%\]
%One
then uses a WKB approximation of the solution and looks for a homogenized system which describes the mean
behavior of the resulting approximation.  Heuristically then, one may expect that a necessary condition for the stability of such solutions is the
hyperbolicity-- i.e., local well-posedness-- of the resulting first order system of partial differential equations.  In order to make this
intuition rigorous, one must study in detail the spectrum of the linearized
operator about a periodic solution in a neighborhood
of the origin in the spectral plane, and compare the resulting low-frequency
stability region to that predicted by hyperbolicity of the formal
Whitham expansion.

%Such analysis was recently conducted by Bronski and Johnson \cite{BrJ}.
The first part of this
analysis was recently conducted by Bronski and Johnson \cite{BrJ}.
There, the authors studied the spectral stability of spatially periodic traveling wave solutions to \eqref{gkdv}
to long-wavelength perturbations by using periodic Evans function techniques, i.e. Floquet theory.  In particular,
an index was derived whose sign, assuming a particular non-degeneracy condition holds, determines the local structure
of the spectrum near the origin in the spectral plane.  This index arises as the discriminant of a polynomial
which encodes the leading order behavior of the Evans function near the origin for such perturbations: that is,
it describes the linearized dispersion relation near zero frequency of the corresponding linearized equations
about the background periodic wave.
Notice that the origin in the spectral plane is distinguished by the fact that the tangent space to the manifold
of traveling wave profiles can be explicitly computed, and the generalized null space of the corresponding linearized
operator is spanned by a particular basis of this tangent space (again, assuming a particular non-degeneracy condition is met).
In particular, Bronski and Johnson showed that if one assumes the conserved quantities of the PDE \eqref{gkdv} provide
good local coordinates for near by periodic traveling waves, then the generalized periodic null space
of the linearized operator is generated by taking specific variations in the parameters defining the periodic
traveling waves.  As a result, the linearized dispersion relation near zero frequency can be explicitly
computed in terms of Jacobians of various maps from the traveling wave parameters to the conserved quantities
of the PDE flow.  These Jacobians were then shown to be explicitly computable in terms of moments
of the underlying wave in the case of power-law nonlinearity $f(u)=u^{p+1}$, $p\in\mathbb{N}$ (see \cite{BrJK}).

The purpose of this paper is to carry out the second part of the program,
connecting the formal Whitham procedure with the rigorous results of \cite{BrJ}.
To motivate our approach, recall that in \cite{W2} Whitham caries out
the formal modulation approximation for the cnoidal wave solutions of the
KdV equation, in which everything can be evaluated explicitly in terms of elliptic integrals.
As a first step in his analysis, Whitham introduces a periodic potential function $\phi$ with the requirement $u=\phi_x$
where $u$ is a given solution of the KdV equation.  In effect, the introduction of this potential allows for
a Lagrangian formulation of the KdV which can can analyzed via the modulation theory presented in \cite{W2}.  However, notice that by requiring the solution
$u$ to be of divergence form one is forcing that the solution have mean zero over one period.  This
of course presents no problem in the case of the KdV equation since all solutions can be made mean zero by Galilean
invariance, but it does pose a serious problem when trying to extend Whitham's method to more general equations of form
\eqref{gkdv} which admit periodic waves of non-zero mean;
%In particular, in order to carry out the modulation approximation presented by Whitham for the generalized
%KdV you must artificially restrict your attention to periodic waves with mean zero, which is highly restrictive:
for example, the approximation would not be valid for the well known elliptic
function solutions of either the focusing or defocusing modified KdV equation.

In order to consider more general equations of the form \eqref{gkdv}, then, it becomes imperative to
find a way to side step this restriction to mean zero waves.  To this end we follow the novel approach suggested by
Serre \cite{S}, which is towork with the original variable $u$ and augment
the resulting WKB system with an additional conservation law associated with the Hamiltonian structure
of the gKdV; see also the recent work of Zumbrun and Oh \cite{OZ3} in which this approach was used
in the viscous conservation law setting.  As we will see, this approach not only works in the KdV case originally considered by Whitham,
but also for the general Hamiltonian gKdV.  In particular, we will see that the linearized dispersion relation predicted
by the linearized Whitham system correctly predicts (assuming particular non-degeneracy conditions are met) the modulational
stability of the given periodic wave.  To this end, we compare the linearized dispersion relation coming from the
modulation approximation to that derived by Bronski and Johnson using Evans function techniques.
%In particular, we validate the program proposed by Serre of deriving the Whitham
%system using additional equations/conserved quantities provided by the Hamiltonian structure of the equations.
Unlike the works of Serre and Zumbrun and Oh  \cite{S,OZ3}, however, we show the equivalence by an essentially different method
using direct computation and identities derived in \cite{BrJ}
as opposed to direct linear-algebraic comparisons.  The remarkable fact
that the resulting homogonized system can be explicitly computed in terms of
the conserved quantities of the PDE flow and the parameters defining the periodic
traveling waves seems to be very special to the case of the gKdV equation \eqref{gkdv}
in which the ODE defining the traveling wave profiles is completely integrable.
We expect that the approach proposed by Serre to show this equivalence should work
in this case as well, by using the fact that additional conservation laws gives rise to additional
first order variations, but the direct method employed here seems simpler in the present context.

The results in this paper therefore extend the original work of Whitham concerning the modulational
stability of cnoidal waves of the KdV to more general periodic waves of equations of form \eqref{gkdv},
and also rigorously validate the predictions of the Whitham expansion
as regards linearized stability about a periodic wave.
However, as pointed out in \cite{OZ3,OZ4}, this
is only the begining of the story.  The importance of the connection
between Whitham expansion and linearized stability comes rather from
the predictive power of the formal expansion in deriving more detailed
information about behavior of solutions.
See also the earlier work of Schneider \cite{S1,S2,S3}
for reaction--diffusion and other equations.
Thus, as in the related contexts \cite{S1,S2,S3,OZ3,OZ4}, we expect the
validation of this connection to lead to further important results
about these and other systems.

The outline of this paper is as follows.  In section 2 we recall the basic properties of the periodic
traveling wave solutions of \eqref{gkdv} described in \cite{BrJ} and \cite{J1}.  In particular, we discuss
a parametrization of such solutions which will be utilized throughout the present work.
In section 3, we outline the results of Bronski and Johnson concerning the low frequency asymptotics
of the periodic Evans function and describe the connection to the modulational stability
of the underlying periodic wave.  In section 4, we carry out the slow modulation (WKB) approximation
to derive the Whitham system for the gKdV equation.  In particular, we connect this formal
homogenization procedure to the low frequency analysis of Bronski and Johnson.  We then close
with a discussion and open problems.

\section{Periodic Solutions of the gKdV}\label{s:prop}

Throughout this paper, we are concerned with the periodic traveling wave solutions of the gKdV equation.
To begin then, we recall the basic properties of the periodic traveling wave solutions of \eqref{gkdv}:
for more details, see \cite{BrJ}, \cite{BrJK}, or \cite{J1}.

Such solutions are stationary solutions of \eqref{gkdv} in a moving coordinate frame of the form $x+ct$
and whose profiles satisfy the traveling wave ordinary differential equation
\begin{equation}
u_{xxx}+f(u)_x-cu_x=0.\label{travelode}
\end{equation}
This profile equation is Hamiltonian and hence can be reduced through two integrations to the nonlinear oscillator
equation
\begin{equation}
\frac{u_x^2}{2}=E+au+\frac{c}{2}u^2-F(u)\label{quad1}
\end{equation}
where $F'=f$, $F(0)=0$, and $a$ and $E$ are constants of integration.  Thus, the existence of periodic orbits of
\eqref{travelode} can verified through simple phase plane analysis: a necessary and sufficient condition is
for the effective potential energy $V(u;a,c)=F(u)-au-\frac{c}{2}u^2$ to have a local minimum.  It follows that the
periodic solutions of \eqref{travelode} form, up to translation\footnote{The translation mode is simply inherited
by the translation invariance of the governing PDE and can hence be modded out in our theory.}, a three parameter family of traveling wave solutions
of \eqref{gkdv} which we can parameterize by the constants $a$, $E$, and $c$.  In particular, on open sets in $\RM^3=(a,E,c)$
the solution to \eqref{quad1} is periodic: the boundary of these open sets corresponds to the equilibrium solutions and solitary waves (homoclinic/heteroclinic orbits).

In addition, we make the assumption that there exist simple roots $u_{\pm}$ of the equation $E=V(u;a,c)$ which
satisfy $u_-<u_+$, and that $V(u;a,c)<E$ for $u\in(u_-,u_+)$.  As a consequence, the roots $u_{\pm}$ are $C^1$
functions of the traveling wave parameters $(a,E,c)$ and, without loss of generality, we can assume that $u(0)=u_{-}$.  It follows
that the period of the corresponding periodic solution of \eqref{travelode} can be expressed by the formula
\begin{equation}
T=T(a,E,c)=\sqrt{2}\int_{u_-}^{u_+}\frac{du}{\sqrt{E-V(u;a,c)}}=\frac{\sqrt{2}}{2}\oint_{\Gamma}\frac{du}{\sqrt{E-V(u;a,c)}},\label{period}
\end{equation}
where integration over $\Gamma$ represents a complete integration from $u_{-}$ to $u_+$, and then back to $u_-$ again: notice
however that the branch of the square root must be chosen appropriately in each direction.  Alternatively, you could interpret
$\Gamma$ as a loop (Jordan curve) in the complex plane which encloses a bounded set containing both $u_-$ and $u_+$.
By a standard procedure, the above integral can be regularized at the square root branch points and hence represents
a $C^1$ function of the traveling wave parameters: see \cite{BrJ} for details.

Notice that in general, the gKdV equation admits three conserved quantities\footnote{In the case of the KdV or mKdV, the complete integrability
of the PDE implies the existence of an infinite number of conservation laws.  We will discuss this case more carefully in the following
sections.} which can be represented as
\begin{align}
M(a,E,c)&=\int_0^Tu(x)dx=\frac{\sqrt{2}}{2}\oint_{\Gamma}\frac{u~du}{\sqrt{E-V(u;a,c)}}\label{mass}\\
P(a,E,c)&=\int_0^Tu(x)^2dx=\frac{\sqrt{2}}{2}\oint_{\Gamma}\frac{u^2~du}{\sqrt{E-V(u;a,c)}}\label{momentum}\\
H(a,E,c)&=\int_0^T\left(\frac{u_x^2}{2}-F(u)\right)dx=\frac{\sqrt{2}}{2}\oint_{\Gamma}\frac{E-V(u;a,c)-F(u)}{\sqrt{E-V(u;a,c)}}du\nonumber
\end{align}
representing the mass, momentum, and Hamiltonian, respectively.  As above, these integrals can be regularized at the
square root branch points and hence represent $C^1$ functions of the traveling wave parameters.  As seen in \cite{BrJ}
and related papers, the gradients of the period, mass, and momentum plan a large role in the stability of the periodic
traveling wave solutions of \eqref{gkdv}.  For notational simplicity then, we introduce the Poisson bracket like notation
\[
\{f,g\}_{x,y}=\det\left(\frac{\partial(f,g)}{\partial(x,y)}\right)
\]
for two-by-two Jacobians, and $\{f,g,h\}_{x,y,z}$ for analogous three-by-three Jacobians.

It should be pointed out that in \cite{BrJ} it was shown when $E\neq 0$ gradients
in the period can be interchanged for gradients of the conserved quantities via the relation
\begin{equation}
E\nabla_{a,E,c}T+a\nabla_{a,E,c}M+\frac{c}{2}\nabla_{a,E,c}+\nabla_{a,E,c}H=0\label{grad}
\end{equation}
where $\nabla_{a,E,c}=\left<\partial_a,\partial_E,\partial_c\right>$.  Thus, although all results of this paper will be expressed
in terms of Jacobians involving the quantities $T$, $M$, and $P$, it is possible to re-express them completely in terms of conserved quantities
of the gKdV flow.  Such an interpretation seems possibly more natural and desired from a physical point of view.

\section{Evans Function Calculations}\label{s:Evans}

Now, suppose $u=u(\cdot;a,E,c)$ represents a particular periodic traveling wave solution of \eqref{gkdv}.
As we are interested in studying the spectral stability of this solution to localized perturbations,
we must study the linearized equation about $u_0$, namely
\[
\partial_x\mathcal{L}[u]v=-v_t
\]
where $v\in L^2(\RM)$ and the operator $\mathcal{L}[u]=-\partial_x^2-f'(u)+c$ is a self adjoint
periodic Hill operator.  Taking the Laplace transform in time yields the linearized spectral problem
\begin{equation}
\partial_x\mathcal{L}[u]v=\mu v\label{spec}
\end{equation}
where $\mu$ represents the Laplace frequency.  We refer to the background solution $u$ as being spectrally stable if the $L^2$ spectrum
of the linear operator $\partial_x\mathcal{L}[u]$ is confined to the imaginary axis: notice the spectrum
is symmetric about the real and imaginary axis.

The spectral analysis in $L^p(\RM)$, $1\leq p<\infty$,
of differential operators with periodic coefficients is known as Floquet theory.  In this theory, one
rewrites the linearized spectral problem \eqref{spec} as a first order system of the form
\begin{equation}
{\bf Y}_x=\HM(x,\mu){\bf Y}\label{system}
\end{equation}
and lets $\Phi(x,\mu)$ be a matrix solution satisfying the initial condition $\Phi(0,\mu)={\bf I}$ for all $\mu\in\CM$, where ${\bf I}$
is the standard $3\times 3$ identity matrix.  The monodromy operator $\MM(\mu)$, or the period map, is then defined to be
\[
\MM(\mu):=\Phi(T,\mu).
\]
Notice that given any vector solution ${\bf Y}$ of the above first order system, the monodromy operator acts via the formula
\[
\MM(\mu){\bf Y}(x,\mu)={\bf Y}(x+T,\mu)
\]
for all $x\in\RM$ and $\mu\in\CM$.  In particular, it follows that %the spectral problem \eqref{spec} can not admit
%any $L^2$ eigenvalues, and that $\mu\in\CM$ belongs to the spectrum $\spec(\partial_x\mathcal{L}[u])$ if and only
%if \eqref{spec} admits a nontrivial uniformly bounded solution.  As a result,
the $L^2$ spectrum of the operator $\partial_x\mathcal{L}[u]$ is purely continuous and $\mu\in\spec(\partial_x\mathcal{L}[u])$ if and only if
\[
\det(\MM(\mu)-e^{i\kappa}{\bf I})=0
\]
for some $\kappa\in\RM$.  Such an equality is equivalent with the existence of a nontrivial solution $v$ of \eqref{gkdv}
such that
\begin{equation}
v(x+mT)=e^{im\kappa}v(x)\label{floquet}
\end{equation}
for all $x\in\RM$ and $m\in\ZM$, which is equivalent with \eqref{spec} admitting a nontrivial uniformly bounded solution.

Following Gardner (see \cite{G1} and \cite{G2}), we define the periodic Evans function for our problem to be
\begin{equation}
D(\mu,\kappa)=\det\left(\MM(\mu)-e^{i\kappa}{\bf I}\right),~~(\mu,\kappa)\in\CM\times\RM.\label{Evanseqn1}
\end{equation}
The complex constant $e^{i\kappa}$ is called the Floquet multiplier and the constant $\kappa$ (which is not uniquely defined)
is called the Floquet exponent\footnote{By \eqref{floquet}, the Floquet exponent encodes a class of admissible perturbations.}.
In particular, by the above discussion we see $\mu\in\spec(\partial_x\mathcal{L}[u])$
if and only if $D(\mu,\kappa)=0$ for some $\kappa\in\RM$.  Moreover, since the coefficient matrix $\HM(\mu,x)$ in \eqref{system}
depends analytically on $\mu$, it follows that $D$ is analytic in the variables $\mu$ and $\kappa$ and is real
whenever $\mu$ is real.

From a computational viewpoint, however, it is more convenient to define the periodic Evans function
by choosing
%An alternative way to write the Evans function is to choose
a basis
$\{\Phi_j(\cdot,\mu)\}_{j=1}^3$ of the kernel\footnote{Here, we are considering the formal
operator $\partial_x\mathcal{L}[u]$ with out any reference to boundary conditions.}
of the linear operator $\partial_x\mathcal{L}[u]-\mu$ and noticing from \eqref{Evanseqn1} that
\begin{equation}
D(\mu,\kappa)=\left(\det(\Phi_j(0,\mu)\big{|}_{j=1}^3)\right)^{-1}\det\left(\Phi_j(T,\mu)-e^{i\kappa}\Phi_j(0,\mu)\big{|}_{j=1}^3\right)\label{Evanseqn2}
\end{equation}
This view is particularly convenient in our case: as a consequence of the integrability of the traveling wave ODE \eqref{travelode}
one can easily verify that the set of functions $\{u_x,u_E,u_a\}$ is linearly independent and satisfy
\[
\mathcal{L}[u]u_x=0,~~\mathcal{L}[u]u_E=0,~~\textrm{and}~~\mathcal{L}[u]u_a=-1.
\]
The first of these is a reflection of the translation invariance of \eqref{gkdv}.
Thus, by taking variations of the underlying periodic solution $u(\cdot;a,E,c)$ of \eqref{travelode}
in the traveling wave parameters generates an explicit basis for the (formal) kernel of $\partial_x\mathcal{L}[u]$.
In particular, one can easily verify that $u_x$ and a linear combination of $u_a$ and
$u_E$ constitute bonafide $T$-periodic elements of the kernel of $\partial_x\mathcal{L}[u]$,
and hence $D(\mu,0)=\mathcal{O}(|\mu|^2)$ for $|\mu|\ll 1$.  Continuing, it follows again
from \eqref{travelode} that $\partial_x\mathcal{L}u_c=-u_x$ and hence an appropriate
linear combination of $u_c$, $u_a$, and $u_E$ lies in the first $T$-periodic Jordan
chain in the translation direction and hence $D(\mu,0)=\mathcal{O}(|\mu|^3)$ for $|\mu|\ll 1$.
Moreover, one can verify that if the Jacobian $\{T,M,P\}_{a,E,c}$ is non-zero, then linear
combinations of $u_x$, $u_a$, $u_E$, and $u_c$ generate the entire generalized null space
of $\partial_x\mathcal{L}[u]$ and hence the Evans function should indeed be of order $\mathcal{O}(|\mu|^3)$
near the origin.

Using perturbation theory one can then analyze the way these solutions bifurcate
from the $(\mu,\kappa)=(0,0)$ state by using \eqref{Evanseqn2} with $\Phi_1(x,0)$, $\Phi_2(x,0)$, and $\Phi_3(x,0)$
corresponding to $u_x$, $u_a$, and $u_E$ respectively.  Moreover, it follows that the quantity
\[
\frac{\partial}{\partial \mu}\Phi_1(x,\mu)\Big|_{\mu=0}
\]
corresponds to a linear combination of $u_c$, $u_a$, and $u_E$, and that the first order variations
in the $u_a$ and $u_E$ directions can be computed via variation of parameters.  The details
of this asymptotic calculation were carried out in Lemma 2 and Theorem 3 of \cite{BrJ},
which we now recall as the main result of this section.

\begin{thm}\label{Evans}
Assume that $\{T,M,P\}_{a,E,c}$ is non-zero.  Then in a neighborhood of $(\mu,\kappa)=(0,0)$ the periodic Evans function admits
the following asymptotic expansion:
\[
D(\mu,\kappa)=-\frac{\mu^3}{2}\{T,M,P\}_{a,E,c}+\frac{i\kappa\mu^2}{2}\left(\{T,P\}_{E,c}+2\{M,P\}_{a,E}\right)+i\kappa^3+\mathcal{O}(|\mu|^4+\kappa^4).
\]
\end{thm}
%
%The proof is given in Lemma 2 and Theorem 3 of \cite{BrJ}.  In particular, since $D(\mu,0)=\mathcal{O}(|\mu|^3)$ for $|\mu|\ll 1$, it follows
%that the linearized operator $\partial_x\mathcal{L}[u]$ has three $T$-periodic eigenvalues at $\mu=0$.
%
%
%
% and the asymptotic
%formula provided by Theorem \ref{Evans} encodes, to leading order, how these three eigenvalues bifurcate from $\mu=0$
%as the Floquet parameter $\kappa$ is varied.

Notice that from Theorem \ref{Evans}, spectral stability in a neighborhood of the origin is equivalent with the cubic polynomial
\begin{equation}
P(y)=-y^3+\frac{y}{2}\left(\{T,P\}_{E,c}+2\{M,P\}_{a,E}\right)-\frac{1}{2}\{T,M,P\}_{a,E,c}\label{delta1}
\end{equation}
having three real solutions, and hence modulational stability can be inferred from the discriminant of the polynomial $P$.
Indeed, if this discriminant is positive then the polynomial $P$ has three real roots and hence, using the Hamiltonian
structure of the linearized operator $\partial_x\mathcal{L}[u]$, the spectrum in a neighborhood of the origin consists
of the a triple covering of the imaginary axis.  If the discriminant is negative, however, then $P$ must have a pair
of complex roots $y_{1,2}$ with non-zero imaginary part and hence the spectrum in a neighborhood of the origin consists of the imaginary
axis with multiplicity one, along with two curves which are tangent to lines through the origin with angle $\arg(iy_{1,2})$
to the imaginary axis.  In particular, the underlying periodic traveling wave is modulationally unstable
in the latter case.  This is the main result of \cite{BrJ} concerning the stability of periodic traveling wave solutions of \eqref{gkdv}
to long wavelength perturbations.

It should be pointed out that by standard abelian integral calculations, the Jacobians arising in Theorem \ref{Evans} were shown in \cite{BrJK} to be explicitly
computable in terms of the traveling wave parameters and moments of the underlying wave over a period in the case
of a power law nonlinearity $f(u)=u^{p+1}$, $p\in\mathbb{N}$.  For example, in the case of the KdV equation
with $f(u)=u^2$ it was shown that the discriminant of $P$ takes the form
\[
\disc(P(\cdot))=\frac{(\alpha_{3,0}T^3+\alpha_{2,1}T^2M+\alpha_{1,2}TM^2+\alpha_{0,3}M^3)^2}{\disc(E-V(\cdot;a,c))}
\]
where the constants $\alpha_{j,3-j}$ represent nonlinear combinations of the traveling wave parameters $a$, $E$, and $c$.
In particular, it follows that the underlying periodic traveling wave is modulationally stable provided
the equation
\[
E-V(y;a,c)=E+ay+\frac{c}{2}y^2-\frac{1}{3}y^3=0
\]
has three real solutions, which is equivalent with the constants $(a,E,c)$ corresponding to a periodic
orbit of the traveling wave equation \eqref{travelode}.  Since every periodic traveling wave solution
of the KdV can be represented in terms of the Jacobbi elliptic function $\cn(x;\gamma)$, it follows
that the cnoidal wave solutions of the KdV are always spectrally stable to long-wavelength
perturbations as predicted by Whitham \cite{W2}.  A similar representation holds in the case of the modified KdV equations
with $f(u)=\pm u^3$ and hence a given periodic solution $u(\cdot;a,E,c)$ of the modified KdV is modulationally
stable provided the polynomial equation
\[
E+ay+\frac{c}{2}y^2-\frac{1}{4}y^4=0
\]
has four real solutions.  While the same procedure can be implemented for other power law nonlinearities
$f(u)=u^{p+1}$ with $p\in\mathbb{N}$, one must resort to numerical studies to compute the desired moments
of the underlying periodic solution.

Finally, it should also be noticed that taking $\kappa=0$ in the low frequency expansion of the Evans
function provided in Theorem \ref{Evans} suggests that the stability index
\[
\sign\left(D(\mu,0)D(\Lambda,0)\right)
\]
for real numbers $0<\mu\ll 1\ll\Lambda<\infty$ is determined\footnote{It was shown in \cite{BrJ} by elementary
asymptotic ODE theory that $D(\Lambda,0)<0$ for $\Lambda>0$ sufficiently large.} precisely by the Jacobian $\{T,M,P\}_{a,E,c}$.
As a result, the underlying periodic traveling wave solution is spectrally unstable to co-periodic perturbations
if $\{T,M,P\}_{a,E,c}<0$, and (see \cite{BrJ}) is spectrally stable to such perturbations provided that $\{T,M,P\}_{a,E,c}$
and $T_E>0$.  In particular, the low-frequency expansion in Theorem \ref{Evans} provides no information for the stability
to co-periodic perturbations in the case where the Jacobian $\{T,M,P\}_{a,E,c}$ is zero, i.e. when the map
\[
(a,E,c)\in\RM^3\to\left(T(a,E,c),M(a,E,c),P(a,E,c)\right)
\]
is a local diffeomorphism.  Thus, by \eqref{grad}, when $E\neq 0$ the nonvanishing of $\{T,M,P\}_{a,E,c}$ is equivalent with the statement
that the conserved quantities of the PDE flow defined by \eqref{gkdv} provide good local coordinates
for nearby periodic traveling waves.  This nondegeneracy condition has been seen in the stability
analysis of periodic gKdV waves in several other contexts: for example, it appears in the nonlinear (orbital) stability
analysis of such solutions \cite{BrJK} \cite{J1} to periodic perturbations, as well in the instability analysis to transverse
perturbations in higher dimensional models \cite{J3}.  In particular, using the elliptic integral methods
\cite{BrJK} it has been seen that the quantity $\{T,M,P\}_{a,E,c}$ is generically nonzero for a wide variety
of nonlinearities, including the most physically relevant examples of the KdV and modified KdV equations.
As we will see in the next section, at the level of the Whitham system
this condition is necessary for the resulting first order system of conservations to be of evolutionary type.

\section{Slow Modulation Approximation}\label{s:slow}

We now complement the rigorous results of the previous section with a formal Whitham theory calculation.
In particular, we want to show that the linearized dispersion relation $P(\cdot)$ in \eqref{delta1}
can be derived through a slow modulation (WKB) approximation, and hence that the formal
homogenization procedures suggested by Whitham \cite{W1} \cite{W2} and Serre \cite{S} correctly
describe the stability of the periodic traveling wave solutions of \eqref{gkdv} to long wavelength
perturbations.  To this end, recall from Section \ref{s:prop} that the gKdV admits a four parameter family $\mathcal{T}$ of
periodic traveling wave solutions for some triple $(a,E,c)$.  We can thus form the quotient space $\mathcal{P}:=\mathcal{T}/\mathcal{R}$
under the relation
\[
\left(u~\mathcal{R}~v\right)\iff \left(\exists \xi\in\RM;v=u(\cdot-\xi)\right).
\]
and we have the class functions
\[
T=T(\dot{u}),~~a=A(\dot{u}),~~E=E(\dot{u}),~~c=C(\dot{u}).
\]
Similarly, since the conserved quantities are translation invariant it follows
that $M$ and $P$ can be interpreted as class functions on $\mathcal{P}$ themselves.

Let $(a_0,E_0,c_0)$ correspond to a particular non-constant periodic solution $u_0$ of \eqref{travelode} in $\mathcal{P}$,
that is
\[
\dot{u}_0(x)=\bar{v}(x;a_0,E_0,c_0)
\]
for some $v\in\mathcal{P}$, where $a_0=A(\dot{u}_0)$, $E_0=E(\dot{u}_0)$, and $c_0=C(\dot{u}_0)$.  In particular, from the discussion in Section
\ref{s:prop} it follows that since $u_0$ is nonconstant the projection
\begin{align*}
u&\mapsto \dot{u}\\
\mathcal{T}&\mapsto\mathcal{P}
\end{align*}
is locally a fibration (where the fibers are circles), and hence $\mathcal{P}$ is locally of dimension three.
Now, consider equation \eqref{gkdv} in the moving coordinate frame $x+c_0t$:
\begin{equation}
u_t=u_{xxx}+f(u)_x-c_0u_x.\label{travelgkdv}
\end{equation}
A periodic solution of the corresponding traveling wave ODE is a stationary
solution of \eqref{travelgkdv}, i.e. is a solution with wave speed $s=0$ in this moving coordinate frame.
Letting $\eps>0$, we now rescale \eqref{travelgkdv} by the change of variables $(x,t)\mapsto(\eps x,\eps t)$ so that the
rescaled equation takes the form
\begin{equation}
u_t=\eps^2u_{xxx}+f(u)_x-c_0u_x,\label{travelgkdv2}
\end{equation}
where, with a slight abuse of notation, $x$ and $t$ now refer to the slow variables $\eps x$ and $\eps t$.

Following Whitham then \cite{W1} \cite{W2}, we consider a solution with a WKB type
approximation of the form
\begin{equation}
u^\eps(x,t)=u^0\left(x,t,\frac{\phi(x,t)}{\eps}\right)+\eps u^1\left(x,t,\frac{\phi(x,t)}{\eps}\right)+\mathcal{O}(\eps^2)\label{wkb}
\end{equation}
where $y\mapsto u^0(x,t,y)$ is an unknown $1$-periodic function.  It follows that the local period of oscillation of the function
$u^0$ is $\eps/\partial_x\phi$, where we assume the unknown phase \emph{a priori} satisfies $\partial_x\phi\neq 0$.
We now plug \eqref{wkb} into \eqref{travelgkdv2} and collect like powers of $\eps$.  The lowest power of $\eps$ present
is $\eps^{-1}$ which leads to the equation
\[
\phi_t\partial_y u^0=\left(\phi_x\partial_y\right)^3u^0+\left(\phi_x\partial_y\right)f(u)-c_0\left(\phi_x\partial_y\right)u^0.
\]
By defining
\begin{equation}
s=-\frac{\phi_t}{\phi_x},~~~\textrm{and}~~~\omega=\phi_x,\label{whithamvar}
\end{equation}
it follows that the $\mathcal{O}(\eps^{-1})$ equation is precisely the traveling wave ODE \eqref{travelode}
in the variable $\omega y$ with wavespeed $c_0-s$; In a similar way as above we may define class functions $S$ and $\Omega$ on $\mathcal{P}$,
respectively, associated to the quantities $s$ and $\omega$.  Moreover, we may choose $u^0(x,t,\cdot)$ to be spatially periodic and satisfy
the nonlinear oscillator equation
\begin{equation}
\frac{\left(u_y^0\right)^2}{2}=E+au^0+\frac{c_0-s}{2}\left(u^0\right)^2-F(u^0)\label{wkbquad}
\end{equation}
where $a=a(x,t)$, $E=E(x,t)$, and $s=s(x,t)$ are independent of $y$.  That is, we identify for each $(x,t)\in\RM^2$ the projection of the function
$y\mapsto u^0(x,t,y)$ into $\mathcal{P}$ as being a periodic traveling wave solution of the gKdV in of form
\[
\dot{u}^0(x,t,y)=\bar{u}(y;A(\dot{u}^0),E(\dot{u}^0),c_0-S(\dot{u}^0))
\]
for some function $\bar{u}\in\mathcal{P}$, where we have used the notation from Section \ref{s:prop}.
In particular, notice that the parameters $(A(\dot{u}^0),E(\dot{u}^0),S(\dot{u}^0))$ defining $\dot{u}^0$ in $\mathcal{P}$
depend on the slow variables $x$ and $t$ introduced above.

Continuing, we find that the $\mathcal{O}(\eps^0)$ equation reads
\[
\partial_t u^0=\partial_x f(u^0)-c\partial_x u^0+\partial_x\left(\phi_x^2\partial_y^2u^0\right)+\partial_y\left(\cdots\right)
\]
Averaging over a single period in $y$ and rescaling we find
\begin{equation}
\partial_t\left(M(\dot{u}^0)\Omega(\dot{u}^0)\right)-\partial_x G(\dot{u}^0)=0\label{whitham1}
\end{equation}
where, using $T=T(\dot{u}^0)$ for convenience,
\[
G(\dot{u}^0)=\frac{1}{T}\int_0^T\left(f(\dot{u}^0)-c\dot{u}^0+\partial_y^2\dot{u}^0\right)dy.
\]
Using \eqref{wkbquad} and the periodicity of $u^0$ in the variable $y$, it follows that
\[
G(\dot{u}^0)=\frac{1}{T}\int_0^T\left(A(\dot{u}^0)-S(\dot{u}^0)\dot{u}^0\right)dy=\left(A-SM\Omega\right)(\dot{u}^0).
\]
Moreover, the Schwarz identity $\phi_{xt}=\phi_{tx}$ provides us with the additional conservation law
\begin{equation}
\partial_t\Omega(\dot{u}^0)+\partial_x\left(S(\dot{u}^0)\Omega(\dot{u}^0)\right)=0.\label{whitham2}
\end{equation}

Now, as noted above, the manifold $\mathcal{P}$ has dimension three and hence
equations \eqref{whitham1} and \eqref{whitham2} do not form
a closed system for the unknown function $\dot{u}^0$.   In \cite{W2} this problem was remedied in the case
of the KdV equation by restricting to mean zero periodic waves, which forces the equation count
to be correct.  However, as noted in the introduction, this requirement can not be implemented in
the general case considered here since the gKdV in general admits periodic waves of non-zero mean.
In order to close the system then, we follow the approach of Serre \cite{S} and augment the above
equations with an additional conservation law arising from the Hamiltonian structure of \eqref{gkdv}.
The choice of this extra conservation law seems somewhat arbitrary where there is more than one to choose; for example,
in the case of the modified KdV where complete integrability implies the existence of infinitely many
conservation laws.  To identify a useful conserved quantity in our calculations, notice from
the Section \ref{s:Evans} that the conserved quantity $P$ must play a significant role in the modulational
stability of the underlying periodic solution.  With this as motivation, we notice from \eqref{travelgkdv}
that
\begin{align}
\left(\frac{u^2}{2}\right)_t&=u\left(u_{xxx}+f(u)_x-c_0u_x\right)\nonumber\\
&=\left(uf(u)+u u_{xx}-F(u)-\frac{c_0}{2}u^2-\frac{(u_y)^2}{2}\right)_x\label{extrawhitham}
\end{align}
Substituting \eqref{wkb} into \eqref{extrawhitham} as before and collecting like powers
of $\eps$, we find that the $\mathcal{O}(\eps^{-1})$ equation is the traveling wave ODE equation
for $u^0$, as before, multiplied by the profile $u^0$, and averaging the $\mathcal{O}(\eps^0)$
equation over a single period in $y$ yields the conservation law
\begin{equation}
\partial_t\left(P(\dot{u}^0)\Omega(\dot{u}^0)\right)-\partial_xQ(\dot{u}^0)=0\label{whitham3}
\end{equation}
where
\begin{equation}
Q(\dot{u}^0)=\frac{2}{T}\int_0^T\left(\dot{u}^0f(\dot{u}^0)+\dot{u}^0 \dot{u}^0_{yy}-F(\dot{u}^0)-\frac{1}{2}\left(\dot{u}^0_y\right)^2-\frac{c_0}{2}(\dot{u}^0)^2\right)dy.\label{Q1}
\end{equation}
Again, using \eqref{wkbquad} one finds that
\[
u^0f(u^0)-F(u^0)=\frac{c_0-s}{2}\left(u_0\right)^2+\frac{\left(u^0_y\right)^2}{2}-E-u^0u_{yy}^0
\]
and hence \eqref{Q1} reduces to
\begin{align}
Q(\dot{u}^0)&=-\frac{2}{T}\int_0^T\left(\frac{S(\dot{u}^0)}{2}\left(\dot{u}^0\right)^2+E(\dot{u}^0)\right)dy\nonumber\\
&=-\left(SP\Omega+2E\right)(\dot{u}^0)\label{Q2}.
\end{align}

The homogenized system \eqref{whitham1}, \eqref{whitham2}, and \eqref{whitham3} is a closed system
of three conservation laws in three unknowns, since $\dot{u}^0$ belongs to $\mathcal{P}$.  In particular,
\eqref{whitham1}-\eqref{whitham3} describe the mean behavior of the WKB approximation \eqref{wkb}.
This system is precisely the Whitham system we were seeking, and it can be written in closed
form as
\begin{equation}
\partial_t\left(M\Omega,P\Omega,\Omega\right)(\dot{u}^0)-\partial_x\left(A-SM\Omega,-SP\Omega-2E,-S\Omega\right)(\dot{u}^0)=0.\label{whithamsystem}
\end{equation}
We now wish to make a connection between the system \eqref{whithamsystem} and the modulational stability of the original solution
$u_0$.  To this end, we make the assumption that the matrix
\begin{equation}
\frac{\partial(\dot{u}_0)}{\partial(a,E,s)}\big{|}_{(a,E,s)=(a_0,E_0,0)}\label{evolution}
\end{equation}
is invertible, which implies that nearby periodic waves in $\mathcal{P}$ can be coordinatized by the traveling wave
parameters $(a,E,c)$ near $(a_0,E_0,c_0)$.
%
%
%the conserved quantities of the flow defined by \eqref{gkdv} provide good local coordinates in $\mathcal{P}$ for the
%nearby periodic traveling waves; this property was previously seen to be important in the analysis
%of Bronski and Johnson \cite{BrJ} outlined in the previous section, as well as the work of Johnson \cite{J1} and Bronski, Johnson, and Kapitula \cite{BrJK}.
In particular, we can parameterize the Whitham
system in terms of the variables $(a,E,s)$ near $(a_0,E_0,0)$ and hence
\begin{equation}
\partial_t\left(M\omega,P\omega,\omega\right)(a,E,s)-\partial_x\left(a-sM\omega,-sP\omega-2E,-s\omega\right)(a,E,s)=0.\label{whithamsystem2}
\end{equation}
where now we have dismissed the use of the class functions on $\mathcal{P}$ previously defined and work directly with the functions $M$, $P$, and $T=\omega^{-1}$
defined on $\RM^3$ defined in section \ref{s:prop}.

We now linearize \eqref{whithamsystem2} about the constant solution $(a,E,s)=(a_0,E_0,0)$ corresponding to the original solution $u_0$
chosen above. To begin, notice from the integral representations \eqref{period}-\eqref{momentum} that
\[
\partial_s\left<M,P,\omega\right>=-\partial_c\left<M,P,\omega\right>,
\]
and hence assuming the quantity
%
%In particular, notice that \eqref{whithamsystem} is of evolution type provided the matrix
%\begin{equation}
%\frac{\partial(M\omega,P\omega,\omega)}{\partial(\dot{u}^0)}\label{evolution}
%\end{equation}
%is nonsingular.  Assuming the matrix $\frac{\partial(\dot{u}^0)}{\partial(a,E,c)}$ is nonsingular, this is equivalent
%with the quantity
\[
\det\left(\frac{\partial(M\omega,P\omega,\omega)}{\partial(a,E,s)}\big{|}_{(a,E,s)=(a_0,E_0,0)}\right)=\frac{1}{T^4}\{T,M,P\}_{a,E,c}(a_0,E_0,c_0)
\]
is nonzero, the corresponding linearized system is of evolutionary type, i.e.the system \eqref{whithamsystem}
is of evolutionary type provided that the conserved quantities of the flow defined
by \eqref{gkdv} provide good local coordinates for the nearby periodic traveling waves.  Recall
from Section \ref{s:prop} that this property is generically true for several physically relevant nonlinearities,
and is a major technical assumption in much of the current work on the stability of such solutions: see again \cite{BrJ}, \cite{BrJK}, \cite{J1},
and \cite{J3}.  Moreover the Cauchy problem for the linearized system is locally well posed provided that it is
hyperbolic, which is equivalent with the characteristic equation
\begin{equation}
\det\left(\mu \frac{\partial(M\omega,P\omega,\omega)}{\partial(a,E,s)}
        -\frac{i\kappa}{T}\frac{\partial\left(a-sM\omega,-sP\omega-2E,-s\omega\right)}{\partial(a,E,s)}\right)(a_0,E_0,0)=0\label{wpoly1}
\end{equation}
having three real roots in the variable $\frac{i\kappa}{\mu T}$, i.e. the matrix % which is equivalent with the matrix
\[
\frac{\partial(M\omega,P\omega,\omega)}{\partial\left(a-sM\omega,-sP\omega-2E,-s\omega\right)}(a_0,E_0,0)
\]
should be uniformly diagonalizable with real eigenvalues.

To make the connection between the above (formal) modulation theory calculation and the rigorous Evans function
theory of the previous section, notice that \eqref{wpoly1} is equivalent to the vanishing of
\begin{equation}
\widetilde{P}(\mu,\kappa;a,E,s)=\det\left(\mu \frac{\partial(M\omega,P\omega,\omega)}{\partial(a,E,s)}
        -\frac{i\kappa}{T}\frac{\partial\left(a-sM\omega,-sP\omega-2E,-s\omega\right)}{\partial(a,E,s)}\right)\label{wpoly2}
\end{equation}
at the point $(a_0,E_0,0)$. % provided that the matrix
%\[
%\frac{\partial(\dot{u}^0)}{\partial(a,E,s)}
%\]
is nonsingular.  Moreover, a straightforward computation shows that
\begin{align*}
\widetilde{P}(\mu,\kappa;a_0,E_0,0)&=\mu^3\det\left(\left(
                          \begin{array}{ccc}
                            M_a\omega & M_E\omega & -M_c\omega \\
                            P_a\omega & P_E\omega & -P_c\omega \\
                            -T_a\omega^2 & -T_E\omega^2 & T_c\omega^2 \\
                          \end{array}
                        \right)
                        -\frac{i\kappa}{\mu T}\left(
                                                \begin{array}{ccc}
                                                  1 & 0 & 0 \\
                                                  0 & -2 & 0 \\
                                                  0 & 0 & -\omega \\
                                                \end{array}
                                              \right)
                        \right)\\
&=\frac{2\mu^3}{T^4}\left(-P\left(-\frac{i\kappa}{\mu}\right)+\left(2M_a-P_E-2T_c\right)\left(\frac{i\kappa}{\mu}\right)^2\right)
\end{align*}
where $P(\cdot)$ is the cubic polynomial given in \eqref{delta1} encoding the modulational stability
of the periodic traveling wave solution $u(\cdot;a_0,E_0,c_0)$.  At first sight, this result is disturbing
as it suggests that the formal WKB/homogenization calculation does not accurately describe the modulational
stability of such solutions.  However, this apparent discrepancy can be easily explained using the
integral representations of the period, mass, and momentum given in \eqref{period}-\eqref{momentum}.  Indeed,
notice from this representation that
\[
M_a=P_E=2T_c=-\frac{\sqrt{2}}{4}\oint_\Gamma\frac{u^2~du}{\left(E-V(u;a,c)\right)^{3/2}}
\]
and hence we have the relation $P_E+2T_c-2M_a=0$.  In particular, it follows that
\begin{equation}
\widetilde{P}(\mu,\kappa;a_0,E_0,0)=-\frac{2\mu^3}{T^4}P\left(-\frac{i\kappa}{\mu}\right).\label{Peqn}
\end{equation}
Notice that the root of the unexpected $\mu\mapsto-\mu$ conversion in \eqref{Peqn} stems from the fact that in the Evans function
calculation described in Section \ref{s:Evans} we considered left moving waves (waves constant in the moving coordinate frame $x+ct$), while in
homogenization calculations of the above form it is customary to use right moving waves.  Indeed, from our definition
of $s$ in \eqref{whithamvar} it follows that the quantity $\phi(x-st)$ is constant, i.e. $\phi$ corresponds
to a \emph{right} moving plane wave.  Since changing right moving waves to left moving waves is equivalent
to time reversal, which in turn arises as a $\mu\mapsto-\mu$ at the level of the linearized equations,
this added negative sign seems necessary in the framework presented here.  In particular, we have proven
the following theorem.

\begin{thm}\label{mainthm}
Assume that $\{T,M,P\}_{a,E,c}$ is non-zero and that the matrix in \eqref{evolution} is nonsingular.
Then there exists a nonzero constant $\Gamma_0$ such that the periodic Evans function admits the following
asymptotic expansion in a neighborhood of $(\mu,\kappa)=(0,0)$:
\begin{align*}
D(\mu,\kappa)&=\Gamma_0\widetilde{P}(-\mu,\kappa;a_0,E_0,0)
%\det\Big(\mu \frac{\partial(M\Omega,P\Omega,\Omega)}{\partial(a,E,s)}
%\\
%        &\quad -\frac{i\kappa}{T}\frac{\partial\left(A-SM\Omega,-SP\Omega-2E,-S\Omega\right)}{\partial(a,E,s)}\Big)\Big{|}_{(a,E,s)=(a_0,E_0,0)}\\
        +\mathcal{O}(|\mu|^4+\kappa^4).
\end{align*}
That is, the linearized dispersion relation arising from the formal WKP approximation correctly describes the true
linearized dispersion relation of the spectrum of the linearization about the underlying periodic traveling wave
at the origin.
\end{thm}

\begin{corr}\label{cor1}
Under the assumptions of Theorem \ref{mainthm}, %, assume moreover that
%\[
%\det\left(\frac{\partial(M\omega,P\omega,\omega)}{\partial(u^0)}\right)\Big{|}_{(a,E,s)=(0,0,0)}\neq 0.
%\]
%Then
a necessary condition for the spectral stability of the periodic traveling wave solution
$u_0(\cdot)=u(\cdot;a_0,E_0,c_0)$ is that the spectrum of the three-by-three matrix
\begin{align*}
&A(a,E,s):=\left(\frac{\partial\left(a-sM\omega,-sP\omega-2E,-s\omega\right)}{\partial(a,E,s)}\right)^{-1}
\left(\frac{\partial(M\omega,P\omega,\omega)}{\partial(a,E,s)}\right)\\
&~~~~~~~~~~~~=\frac{\partial(M\omega,P\omega,\omega)}{\partial\left(a-sM\omega,-sP\omega-2E,-s\omega\right)}
\end{align*}
be real at $(a_0,E_0,0)$: equivalently, that the Whitham system \eqref{whithamsystem2} be hyperbolic
at $(a_0,E_0,0)$.
\end{corr}

\begin{proof}
First, a straight forward calculation shows that the matrix
\[
\frac{\partial\left(a+sM\omega,sP\omega-2E,-s\omega\right)}{\partial(a,E,s)}\Big{|}_{(a,E,s)=(a_0,E_0,0)}
\]
is invertible provided $\{T,M,P\}_{a,E,c}$ is non-zero at $(a,E,c)=(a_0,E_0,c_0)$ and the matrix in \eqref{evolution} is nonsingular
at this special point.  Moreover, if the matrix $A(a_0,E_0,0)$ had a nonzero eigenvalue $\eta$, then it follows from Theorem \ref{mainthm} that there
is a branch of spectrum bifurcating from the origin admitting and asymptotic expansion
\[
\mu=\frac{i\eta\kappa}{T}+\mathcal{O}(\kappa^2)
\]
for $|\kappa|\ll 1$.  Thus, if $\eta$ has non-zero imaginary part one immediately has a modulational instability of the underlying
periodic traveling wave solution.
\end{proof}

It follows that we have established that a necessary condition for the spectral stability of a periodic traveling wave solution
of \eqref{gkdv} is that the Whitham system \eqref{whitham1}, \eqref{whitham2}, and \eqref{whitham3} be hyperbolic at the
corresponding solution.  Moreover, from the analysis in \cite{BrJ}, if the matrix $A(a_0,E_0,0)$ has only simple real eigenvalues, it follows
that the underlying periodic wave is modulationally stable.  Indeed, the Hamiltonian structure of the linearized operator
$\partial_x\mathcal{L}[u]$ implies the spectrum is symmetric about the imaginary and real axes, and hence if
the matrix $A(a_0,E_0,0)$ has real distinct eigenvalues $\alpha_1$, $\alpha_2$, and $\alpha_3$, then there are three branches $\mu_j(\kappa)$
which bifurcate from the origin with leading order expansions $\mu_j(\kappa)=\frac{i\alpha_j\kappa}{T}+\mathcal{O}(\kappa^2)$.
If one did not have spectral stability near the origin, then two of the branches, %must coincide, say $\mu_1(\kappa)=\mu_2(\kappa)$,
say $\mu_1$ and $\mu_2$, would satisfy $\mu_1(\kappa)=-\overline{\mu_2(\kappa)}$.
It follows by equating like powers of $\kappa$ that $\alpha_1=\alpha_2$, contradicting the simplicity of
the spectrum of $A(a_0,E_0,0)$.  Thus, the strict hyperbolicity of the Whitham system is sufficient to conclude modulational stability
of the underlying wave.
%
%
%
%that modulational stability of the periodic traveling wave is equivalent with the hyperbolicity of the Whitham system at
%the underlying solution, and hence concerning modulational stability one has a partial converse to Corollary \ref{cor1} provided by the
%rigorous Evans function approach.

%
%\br
%The above-mentioned result on stability is somewhat remarkable,
%depending on additional structure of the model.
%In general, well-posedness of the Whitham system can be expected
%to yield only information on {\rm long-wave}, or modulational stability,
%and not stability to high-frequency perturbations.
%\er

%\section{Conclusions \& Discussion}
\section{Discussion and open problems}

In this paper, we have rigorously verified that the formal homogenization procedure introduced by Whitham to study
the modulational stability of a periodic traveling wave profile of the gKdV equation does indeed describe
the spectral stability near the origin.
This in particular applies to the KdV case for which the WKB
expansion is carried out in Chapter 14 of Whitham \cite{W2} by
Lagrangian methods making use of integrability to carry out
computations in terms of elliptic integrals.
(The rigorous connection between the Lagrangian formulation and the
type of WKB expansion carried out here is an important result of \cite{W2}.)

Recall from Section \ref{s:slow} that there is at first sight
a ``missing equation'' in the WKB equation for (KdV) and (gKdV).
This is remedied in \cite{W2} by introducing a ``potential'' consisting
of the anti-derivative of the solution, thus viewing the equation as
a one higher order PDE for which the equation count is correct.
Here, we follow instead the approach suggested by Serre \cite{S} of augmenting
the system by an additional conservation law associated with the integral
of motion coming from Hamiltonian structure.
This approach works not only in the integrable KdV and mKdV cases, but also
for the general Hamiltonian (gKdV).
The derivation of the Whitham equations seems also slightly simpler from
this point of view.  Indeed, using this approach of Serre together with
identities developed by Bronski and Johnson \cite{BrJ}, it seems remarkable
that we are able to obtain the Whitham system \eqref{whithamsystem} explicitly in terms
of moments of the underlying periodic wave and the traveling wave parameters.  Because
of this, the verification that the hyperbolicity of the Whitham system correctly describes
the spectral stability near the origin proved to be much simpler and straight forward
than cases previously considered by Oh and Zumbrun \cite{OZ3} and Serre.

This issue of a missing equation comes up when there are more periodic
solutions than would generically expected \cite{S}, as arises when
there is an integral of motion for the traveling wave profile ODE.
Serre states on p. 262 of \cite{S} that
``a supplementary integral has a counterpart at the profile level'',
that is, such ODE integrals of motion often come from conserved quantities
The ``supplementary integral'' he refers to is exactly an additional
conservation law, which he proposes to use as above
to fill the same missing equation it causes.
That is, this same basic approach, by the argument of Serre, should
work in a wide range of situations.

In situations such as the KdV, where there may be still further conservation
laws, there is flexibility in deriving the Whitham system, but all such
derived systems must be equivalent.  This would appear to
correspond to existence for the Whitham system of additional conservation laws, or hyperbolic ``entropies'', another
interesting possibility pointed out by Serre.

We here just briefly investigate further the
statement of Serre that supplementary PDE integrals
should imply also supplementary ODE integrals.
We find justification for this
statement in two very general situations.

1. (more general). If by a supplementary integral, we mean an ``entropy''
in the conservation law sense, i.e., a function $\eta(u)$ of the
unknown and no derivatives,
such that $\eta(u)= \partial_x(...)$, then this is indeed always true.
(Note: the quantity $\int \eta(u(x))dx$ is then conserved in time.)

\begin{proof}
Proof:   we then have additional traveling wave relation
$-s\eta(u)+q(u)+ C = L(u,u_x,...),$
where $C$ is a constand of integration and $L$ is an expression
of the same order as the traveling-wave ODE.
But, combined with existing equations, we can then eliminate the highest
derivative to get an algebraic relation between some lower derivatives,
i.e., and integral of motion of the traveling-wave profile ODE.
\end{proof}

2. In the Hamiltonian case, both properties can be inherited from an overarching
Lagrangian structure, as in the multi-symplectic form system studied
by many authors; see for example \cite{BD1,BD2}.
This may tie in an interesting way to the variational approach of Witham.

\medskip
We cite as open problems the rigorous
verification of the Witham approximation at the level of behavior,
either in the small $\eps$ limit as in \cite{GMWZ}, or in large-time
behavior as in \cite{S1,S2,S3,OZ4}.  Moreover, it is expected that a similar homogenization procedure
can be used to justify the Whitham equations for the generalized Benjamin-Bona-Mahony equation, where
the low frequency asymptotics were recently computed in a similar way as to the gKdV case: see \cite{J2}.

\medbreak
{\bf Acknowledgement.}
Many thanks to Jared Bronski for providing several useful insights into the connection of this work to that
presented in \cite{W2} by Whitham.  Also, we thank the reviewer for a careful and detailed check of the formulas
and derivations presented in section four.

%%%%%%%%%%%%%%%%%%%%%%%%%%%%%%%%%%%%%%%%%%%%%%%%%%%%%%%%%%%%%%%%%%%%%%%%%%%%%%%%%%%%%%%%%%%%%%%%%%%%%%%%%%%%%%%%

\end{document}